\documentclass{article}
\usepackage{graphicx} 
\usepackage{amssymb}
\usepackage{geometry}
\usepackage{amsmath}
\usepackage{amsthm}
\usepackage{youngtab}
\usepackage{tikz}
\usepackage{array}
\usepackage{tikz-cd}
\usepackage{hyperref}
\usepackage{verbatim}
\usepackage{multirow}
\usepackage[utf8]{inputenc}
\newtheorem{theorem}{Theorem}[section]
\newtheorem{corollary}[theorem]{Corollary}
\newtheorem{lemma}[theorem]{Lemma}
\newtheorem{proposition}[theorem]{Proposition}
\newtheorem{conjecture}[theorem]{Conjecture}
\newtheorem{question}[theorem]{Question}

\theoremstyle{definition}
\newtheorem{definition}[theorem]{Definition}
 
\theoremstyle{definition}
\newtheorem{remark}[theorem]{Remark}

\theoremstyle{definition}
\newtheorem{example}[theorem]{Example}

\numberwithin{equation}{section}
\newcolumntype{V}{>{\centering\arraybackslash} m{.095 \linewidth} }
\newcommand{\av}{\mathbf{a}}
\newcommand{\bv}{\mathbf{b}}
\newcommand{\ab}{\av\bv}

\DeclareMathOperator{\Id}{Id}

\DeclareMathOperator{\RG}{RG}

\begin{document}
\title{Box polynomials and the excedance matrix}
\author{Richard Ehrenborg, Alex Happ, Dustin Hedmark, and Cyrus Hettle}
\date{}
\maketitle
\begin{abstract}
We consider properties of the box polynomials, a one variable polynomial defined over all integer
partitions $\lambda$ whose 
Young diagrams fit in an $m$ by $n$ box.
We show that these polynomials can be expressed by
the finite difference operator applied to the power $x^{m+n}$. 
Evaluating box polynomials yields a variety of 
identities involving set partition enumeration. 
We extend the latter identities using restricted 
growth words and a new operator called the fast 
Fourier operator, and consider connections 
between set partition enumeration and the 
chromatic polynomial on graphs. We also give 
connections between the box polynomials and the 
excedance matrix, which encodes combinatorial 
data from a noncommutative quotient algebra 
motivated by the recurrence for the excedance 
set statistic on permutations.
\end{abstract}

\section{Introduction}
\label{section_box_polynomial_intro}

In this paper we examine a one-variable polynomial $B_{m,n}(x)$, which
we call the box polynomial.  It is defined by summing over all integer
partitions whose Young diagrams fit in an $m$ by $n$ box (or grid).
We show that this polynomial is related to set partition enumeration
in several ways.  We also obtain bounds for the roots of the box
polynomial. Furthermore, it is also related to the excedance set
statistics from permutation enumeration.

In Section~\ref{section_box_polynomial} we derive basic properties of the box polynomial,
including an alternative description as a repeated application of the
difference operator to a power.  In Section~\ref{section_box_polynomial_chromatic} we show how the
chromatic polynomial of a graph determines the number of set
partitions on the vertex set such that the blocks are independent
sets.  Using the chromatic polynomial of a cycle we obtain that the
number of set partitions such that $i$ and $i+1$ modulo $n$ are in
different blocks is given by a box polynomial evaluated at $-1$.  We
also use the chromatic polynomial of other graphs to obtain
enumerative results for set partitions.  This yields an interpretation
for the box polynomial evaluated at the positive integers.  In
Section~\ref{section_box_polynomial_generating_functions} we use generating functions and restricted growth words to
obtain the results. Here we show that the box polynomial evaluated at
$x=-n/2$ enumerates partitions where the block sizes are odd.
Section~\ref{section_bijective_interlude} is an interlude where we offer a few bijections.
In Section~\ref{section_fast_fourier} we introduce a polynomial operator that we call the Fast
Fourier operator. We show that the Fast Fourier Operator has
connections to both the box polynomials and set partition
enumeration.  Section~\ref{section_box_polynomial_bounds_on_roots} is dedicated to the roots of the box
polynomial.  By using the difference operator description of the box
polynomial and a result of P\'olya, the roots all lie on a vertical
line in the complex plane.  We also bound the roots.  In Section~\ref{section_excedance_matrix} we
review the excedance set statistic and how it connects with the box
polynomials.

We end in the concluding remarks by offering a plethora of open
questions about box polynomials, set partition enumeration and the
excedance set statistic.

\section{Box polynomials}
\label{section_box_polynomial}

We begin by introducing the box polynomials.
\begin{definition}
The {\em box polynomial} $B_{m,n}(x)$ is defined by the sum
\[
B_{m,n}(x)
=
\sum_{\lambda \subseteq m \times n} \prod_{i=1}^{m}(x+\lambda_i) ,
\]
where the sum is over all partitions
$\lambda = (n \geq \lambda_{1} \geq \lambda_{2} \geq \cdots
\geq \lambda_{m} \geq 0)$,
that is, all partitions with $m$ nonnegative parts, each at most $n$.
\label{definition_box}
\end{definition}
\begin{example}
\label{example_box_2_2}
Let $m=n=2$. Using Table~\ref{table_box_definition} we see that 
$B_{2,2}(x)=6x^2+12x+7$.

\end{example}

\begin{table}
\begin{center}
{\renewcommand{\arraystretch}{1.3}
\begin{tabular}{ V | l }
  $\lambda$ & \hspace*{17mm}$\prod_{i=1}^{m}(x+\lambda_i)$\\
  \hline\\[-12pt]
 \begin{tikzpicture}[scale=0.45]
	\draw[gray!40] (0,0) grid (2,2);
\end{tikzpicture} 
									& $(x+0)\cdot (x+0) = x^2$\\
   \begin{tikzpicture}[scale=0.45]
	\draw[gray!40] (0,0) grid (2,2);
	\draw (0,1) grid (1,2);
\end{tikzpicture} 
									& $(x+1)\cdot (x+0)=x^2 +x$  \\
   \begin{tikzpicture}[scale=0.45]
	\draw[gray!40] (0,0) grid (2,2);
	\draw (0,0) grid (1,2);
\end{tikzpicture} 
									& $(x+1)\cdot(x+1)=x^2+2x+1$\\
   \begin{tikzpicture}[scale=0.45]
	\draw[gray!40] (0,0) grid (2,2);
	\draw (0,1) grid (2,2);
\end{tikzpicture} 
									& $(x+2)\cdot (x+0)=x^2+2x$  \\
  \begin{tikzpicture}[scale=0.45]
	\draw[gray!40] (0,0) grid (2,2);
	\draw (0,0) grid (1,2) grid (2,1);
\end{tikzpicture} 
									& $(x+2)\cdot (x+1)=x^2+3x+2$  \\
   \begin{tikzpicture}[scale=0.45]
	\draw[gray!40] (0,0) grid (2,2);
	\draw (0,0) grid (2,2);
\end{tikzpicture} 
									& $(x+2)\cdot (x+2)=x^2+4x+4$ \\  
      \hline & \hspace*{12.7mm}$B_{2,2}(x)=6x^2+12x+7$
\end{tabular}
}

\end{center}
\caption{The box polynomial $B_{2,2}(x)$. The table lists all partitions $\lambda$ that fit in the $2\times 2$ box, upper left justified.}
\label{table_box_definition}
\end{table}
Another way to express the box polynomials
is in terms of the complete symmetric function $h_{m}$.
It follows from Definition~\ref{definition_box} that
\begin{equation}
B_{m,n}(x)
=
h_{m}(x,x+1, \ldots, x+n).
\label{equation_box_stirling}
\end{equation}
Directly from this equation we have that
the box polynomial evaluated at $x=0$ 
and $x=1$ yields Stirling numbers of the second kind,
that is,
\begin{align}
B_{m,n}(0) & = S(m+n,n) , 
\label{equation_box_polynomial_at_zero} \\
B_{m,n}(1) & = S(m+n+1,n+1) .
\label{equation_box_polynomial_at_one}
\end{align}

\begin{lemma}
The box polynomials satisfy
$B_{m,n}(-n-x) = (-1)^{m} \cdot B_{m,n}(x)$.
\label{lemma_even_odd}
\end{lemma}
\begin{proof}
Follows from equation~\eqref{equation_box_stirling}
using that the complete symmetric function $h_{m}$
is homogeneous of degree $m$.
\end{proof}

Directly from equations~\eqref{equation_box_polynomial_at_zero}
and~\eqref{equation_box_polynomial_at_one} we have:
\begin{align}
B_{m,n}(-n) & = (-1)^{m} \cdot S(m+n,n) ,
\label{equation_box_polynomial_at_minus_n}\\
B_{m,n}(-n-1) & = (-1)^{m} \cdot S(m+n+1,n+1) .
\label{equation_box_polynomial_at_minus_n_minus_one}
\end{align}

Furthermore, with equation~\eqref{equation_box_stirling}
we obtain the generating function
\begin{align}
\sum_{m \geq 0} B_{m,n}(x) \cdot t^{m}
& =
\frac{1}{(1 - x \cdot t) \cdot (1 - (x+1) \cdot t) \cdots (1 - (x+n) \cdot t)} . 
\label{equation_box_generating_function}
\end{align}

Note that the box polynomial $B_{m,n}(x)$ has degree $m$,
and the sum defining this polynomial has $\binom{m+n}{m}$ terms, since each partition
$\lambda$ fitting in the $m$ by $n$ box can be
specified uniquely by a lattice walk from
$(0,0)$ to $(m,n)$ with east and north steps. 
This observation shows that the leading coefficient
of the box polynomial $B_{m,n}(x)$ is given 
by $[x^m]B_{m,n}(x)=\binom{m+n}{m}$.

Just as 
the binomial coefficients
satisfy the Pascal recursion
we have the following recursion for 
the box polynomials.
\begin{proposition}
The box polynomial $B_{m,n}(x)$ satisfies the recursion
\[
B_{m,n}(x)
=
x \cdot B_{m-1,n}(x)
+
B_{m,n-1}(x+1) ,
\]
with initial conditions $B_{m,0}(x) = x^{m}$ and $B_{0,n}(x) = 1$.
\label{proposition_box_recursion}
\end{proposition}
\begin{proof}
The initial conditions are straightforward to verify.
The recursion follows the same reasoning as the Pascal recursion.
The last part $\lambda_{m}$ of the partition $\lambda$ fitting in the $m$ by $n$ grid 
is either $0$ or it is greater than or equal to~$1$.
In the first case we have
$\lambda = \mu\circ (0)$ where $\mu$ is a partition
contained in a $(m-1) \times n$ box
and $\circ$ denotes concatenation.
Here we have
$\prod_{i=1}^{m} (x+\lambda_{i})
=
x \cdot \prod_{i=1}^{m-1} (x+\mu_{i})$.
Summing over all $\mu$ yields
$x \cdot B_{m-1,n}(x)$.
In the second case,
$\lambda = (\nu_{1}+1, \nu_{2}+1, \ldots, \nu_{m}+1)$
where $\nu$ is  contained in a $m \times (n-1)$ box.
Now
$\prod_{i=1}^{m} (x+\lambda_{i})
=
\prod_{i=1}^{m-1} (x+1+\nu_{i})$
and summing over all $\nu$ yields
$ B_{m,n-1}(x+1)$.
\end{proof}

We continue by giving 
a different expression of the box polynomials, namely as the image of
the forward difference operator.
We begin by defining the relevant 
polynomial operators. 
Let $E^{a}$ be the shift operator given by 
$E^{a}(p(x))=p(x+a)$
and for brevity we write $E$ for $E^{1}$.
Let $\Delta$ be the forward difference operator
defined by $\Delta=E-\Id$,
so, $\Delta(p(x))=p(x+1)-p(x)$.
Note that the difference operator
is shift invariant, that is,
$\Delta E^{a} = E^{a} \Delta$.
Finally, let
$\underline{x}$ be the operator which multiplies 
by $x$, that is,
$\underline{x}(p(x)) = x \cdot p(x)$.

\begin{lemma}
For a non-negative integer~$n$,
the following identity holds 
$\Delta^{n} \underline{x}
=
\underline{x} \Delta^{n}
+
n \cdot E \Delta^{n-1}.$
\label{lemma_Delta_x}
\end{lemma}
\begin{proof}
When $n=0$ there is nothing to prove.
Begin by observing that
\[\Delta(x \cdot p(x))
= (x+1) \cdot p(x+1) - x \cdot p(x)
= x \cdot \bigl(p(x+1) - p(x)\bigr) + p(x+1)
= \underline{x} \Delta(p(x))
+ E(p(x)),\] and hence the identity holds
for $n=1$. The general case follows
by induction using the case $n=1$
as the induction step.
\end{proof}

Note that the relation
$\Delta \underline{x}  = \underline{x} \Delta + E$
is reminiscent of the Weyl relation
$\frac{d}{dx} \underline{x}  = \underline{x} \frac{d}{dx} + \Id$.

We now give the operator interpretation of the box polynomials.
\begin{theorem}
The box polynomial $B_{m,n}(x)$ satisfies
$B_{m,n}(x) = \Delta^n(x^{m+n})/n!$.
\label{theorem_operator_interpretation}
\end{theorem}
\begin{proof}
The proof is by induction on $m$
and $n$.
The base case $m=0$ or $n=0$ is straightforward.
Using the recursion of Proposition~\ref{proposition_box_recursion}
and the induction hypothesis
we have that
\begin{align*}
B_{m,n}(x)
& =
x \cdot B_{m-1,n}(x) + B_{m,n-1}(x+1)
\\
& =
\underline{x}
(B_{m-1,n}(x))
+
E
(B_{m,n-1}(x)) \\
& =
1/n! \cdot 
\underline{x}
\Delta^{n}(x^{m-1+n})
+
1/(n-1)! \cdot 
E
\Delta^{n-1}(x^{m+n-1}) \\
& =
1/n!
\cdot
(\underline{x} \Delta^{n} +
n \cdot E \Delta^{n-1})
(x^{m+n-1}) \\
& =
1/n!
\cdot
\Delta^{n} \underline{x}
(x^{m+n-1}) ,
\end{align*}
where the last step is
Lemma~\ref{lemma_Delta_x},
completing the induction.
\end{proof}

A different and direct proof of
Theorem~\ref{theorem_operator_interpretation}
is as follows.
\begin{proof}[Second proof of
Theorem~\ref{theorem_operator_interpretation}.]
Using the relation
$\Delta \underline{x}
= \underline{x} \Delta + E$
each occurrence of $\Delta$
in $\Delta^{n} \underline{x}^{m+n}(1)$
can be moved to the right
until it either cancels
an $\underline{x}$
and the pair becomes
a shift operator $E$,
or it moves past each of the $m+n$ occurrences
of $\underline{x}$. 
Because $\Delta(1)$ is zero,
each $\Delta$ is forced to cancel with some
$\underline{x}$ to produce a shift operator $E$. 
Since the order of
the $n$ $\underline{x}$'s
that become
the shift operator $E$
does not matter, we divide by
$n!$ on both sides and obtain
that
\[
1/n! \cdot
\Delta^{n} \underline{x}^{m+n}(1)
=
\sum_{p_{0}+p_{1}+\cdots+p_{n}=m}
\underline{x}^{p_{0}}
E
\underline{x}^{p_{1}}
E
\cdots
E
\underline{x}^{p_{n}}
(1) .
\]
Let $\lambda$ be the partition
which has $p_{i}$ parts
equal to $i$.
Then the term
$\underline{x}^{p_{0}}
E
\underline{x}^{p_{1}}
E
\cdots
E
\underline{x}^{p_{n}}
(1)$
is indeed the product
$\prod_{j=1}^{m} (x+\lambda_{i})$
and 
the result follows by observing
that the
condition
$p_{0}+p_{1}+\cdots+p_{n}=m$
is equivalent to
the partition $\lambda$ satisfying  $\lambda\subseteq m\times n$.
\end{proof}

\begin{lemma}
The derivative of the box polynomial $B_{m,n}(x)$
satisfies
$$
\frac{d}{dx} B_{m,n}(x) = (m+n) \cdot B_{m-1,n}(x) .
$$
\label{lemma_D}
\end{lemma}
\begin{proof}
The derivative operator $\frac{d}{dx}$
commutes with the difference operator $\Delta$. Therefore,
\begin{align*}
\frac{d}{dx} B_{m,n}(x)
&=
\frac{d}{dx}\frac{1}{n!} \cdot \Delta^{n}(x^{m+n})
=
(m+n) \cdot \frac{1}{n!} \cdot \Delta^{n}(x^{(m-1)+n})
=
(m+n) \cdot B_{m-1,n}(x)
\qedhere.
\end{align*}
\end{proof}

Alternatively, Definition~\ref{definition_box}
of the box polynomials can be used
to prove Lemma~\ref{lemma_D}.
Let $\lambda$ be a partition contained in the $m \times n$ box.
Removing one arbitrary entry yields
a partition $\mu$ in a $(m-1) \times n$ box.
Let us write this relationship as $\lambda \sim \mu$.
Note that given $\mu \subseteq (m-1) \times n$
there are $m+n$ possible partitions $\lambda$ such
that 
$\lambda \sim \mu$,
as there are $m+n$ possible entries to insert in $\mu$.
Now apply the product rule to
Definition~\ref{definition_box} and change the order of summation:
\begin{align*}
\frac{d}{dx}
B_{m,n}(x)
& =
\sum_{\lambda \subseteq m \times n}
\frac{d}{dx}
\prod_{i=1}^{m}(x+\lambda_{i}) \\
& =
\sum_{\lambda \subseteq m \times n}
\sum_{\lambda \sim \mu}
\prod_{i=1}^{m-1}(x+\mu_{i}) \\
& =
\sum_{\mu \subseteq (m-1) \times n}
\sum_{\lambda \sim \mu}
\prod_{i=1}^{m-1}(x+\mu_{i}) \\
& =
(m+n) \cdot
B_{m-1,n}(x) .
\end{align*}

We now give another form for the box polynomials.
\begin{proposition}
The box polynomial $B_{m,n}(x)$ is given by the sum
\[B_{m,n}(x)
=
\sum_{j=0}^{m}
\binom{m+n}{j} \cdot S(m+n-j,n) \cdot
x^{j} . \]
\label{proposition_closed_form_stirling_binomial}
\end{proposition}
\begin{proof}
Since the box polynomial $B_{m,n}(x)$ has degree $m$
it is enough to determine the coefficient of~$x^{j}$
for $0 \leq j \leq m$:
\begin{align*}
[x^{j}] B_{m,n}(x)
& =
\frac{1}{j!}
\cdot
\left.
\frac{d^{j}}{dx^{j}}
B_{m,n}(x) \right|_{x = 0}
=
\left.
\frac{(m+n)_{(j)}}{j!} \cdot
B_{m-j,n}(x) \right|_{x = 0}
=
\binom{m+n}{j} \cdot S(m+n-j,n) .
\qedhere
\end{align*}
\end{proof}

\begin{proof}[Second proof of Proposition~\ref{proposition_closed_form_stirling_binomial}]
By the binomial theorem applied to $\Delta=E-\Id$,
the $n$th power of
the difference operator
$\Delta^{n}$
is given by
$\sum_{r=0}^{n}
(-1)^{n-r}
\cdot 
\binom{n}{r}
\cdot
E^{r}$.
Therefore, with Theorem~\ref{theorem_operator_interpretation} we
obtain
\begin{align}
B_{m,n}(x)
& =
\frac{1}{n!} \cdot
\sum_{r=0}^{n}
(-1)^{n-r} \cdot
\binom{n}{r} \cdot
(x+r)^{m+n} 
\label{equation_box_via_Delta} \\
& =
\frac{1}{n!} \cdot
\sum_{r=0}^{n}
(-1)^{n-r} \cdot
\binom{n}{r} \cdot
\sum_{j=0}^{m+n}
\binom{m+n}{j} \cdot
r^{m+n-j} \cdot
x^{j}
\nonumber \\
& =
\sum_{j=0}^{m+n}
\binom{m+n}{j} \cdot
\frac{1}{n!} \cdot
\sum_{r=0}^{n}
(-1)^{n-r} \cdot
\binom{n}{r} \cdot
r^{m+n-j} \cdot
x^{j}
\nonumber \\
& =
\sum_{j=0}^{m+n}
\binom{m+n}{j} \cdot
S(m+n-j,n)
\cdot
x^{j} ,
\nonumber
\end{align}
where the last step used
a classical identity for the Stirling numbers, see~\cite[equation~1.94(a)]{EC1}
or equation~\eqref{equation_chromatic_classic_stirling}.
Note that $S(m+n-j,n)$
is zero when $j>m$ and hence
the upper bound of the last sum
is actually~$m$.
\end{proof}

Using equation~\eqref{equation_box_polynomial_at_one}
and the same idea as the first proof
of Proposition~\ref{proposition_closed_form_stirling_binomial}, we obtain the following identity
$$
B_{m,n}(x)
=
\sum_{j=0}^{m}
\binom{m+n}{j} \cdot
S(m+n-j+1,n+1) \cdot
(x-1)^{j} .
$$

\begin{lemma}
For non-negative integers $m$, $n_{1}$ and $n_{2}$
we have the identity
$$  B_{m,n_{1} + n_{2} + 1}(x)
=
\sum_{k=0}^{m} 
B_{k,n_{1}}(x) \cdot B_{m-k,n_{2}}(x+n_{1}+1) .
$$
\label{lemma_split}
\end{lemma}
\begin{proof}
Any partition $\lambda \subseteq m \times (n_{1}+n_{2}+1)$
can be written uniquely as the concatenation of
the two partitions 
$\nu + n_{1} + 1$ and $\mu$,
where
$\nu \subseteq (m-k) \times n_{2}$,
$\mu \subseteq k \times n_{1}$,
and 
$\nu + n_{1} + 1$ shifts every entry of $\nu$ by $n_{1}+1$.
By summing over all possibilities the identity follows.
\end{proof}

\section{Connections with set partitions via chromatic polynomials}
\label{section_box_polynomial_chromatic}

Earlier we observed that
the box polynomial $B_{m,n}(x)$ evaluated at $x$
equal to $0$, $1$, $-n$ and $-n-1$
yields Stirling numbers of the second kind,
which enumerate set partitions.
In this section we consider other evaluations of the
box polynomial that also enumerate various flavors of set partitions.

For a graph $G$
let $S(G,k)$ be the number of set partitions of the vertex set of $G$
into $k$ blocks such that adjacent vertices of $G$ 
are in different blocks.
Another way to state this is that 
each block is an independent set of the graph $G$.

\begin{theorem}
Let $G$ be a graph on the vertex set $[n] = \{1,2, \ldots, n\}$.
Then the number of set partitions of $[n]$
into $k$ blocks such that adjacent vertices of $G$ 
are in different blocks
is given by
\[
S(G,k)
=
\frac{1}{k!}
\cdot
\Delta^{k}( \chi(G;x) )\rvert_{x=0},
\]
where $\chi(G;x)$ is the chromatic polynomial of the
graph $G$.
\label{theorem_chromatic}
\end{theorem}
\begin{proof}
Consider a legal coloring of the graph $G$, that is, a function
$f : [n] \longrightarrow [x]$ such that
for an edge $ij$ we have $f(i) \neq f(j)$.
By defining the blocks $C_{r} = \{i \in [n] \: : \: f(i) = r\}$,
we obtain an ordered set partition
$(C_{1}, C_{2}, \ldots, C_{x})$,
where the blocks are independent sets and possible empty.
Hence the chromatic polynomial $\chi(G;x)$ enumerates these
ordered set partitions of $[n]$ into $x$ possibly empty blocks.
By inclusion-exclusion, the number of ordered sets
partitions into $k$ blocks where the blocks are non-empty independent sets
is given by the alternating sum
$\sum_{i=0}^{k}   (-1)^{k-i} \cdot \binom{k}{i} \cdot \chi(G;i)$,
since $\binom{k}{i}\cdot\chi(G;i)$ counts set partitions 
of $[n]$ into $k$ parts with blocks forming independent
sets of $G$ with at least $k-i$ empty blocks.

The result follows by removing the order between the blocks,
that is, dividing by $k!$. Finally, express the result in terms of
the forward difference operator $\Delta$ applied $k$ times.
\end{proof}

Note that the empty graph on $n$ vertices, that is the graph with no 
edges, has chromatic polynomial~$x^{n}$. Since any subset
of vertices of the empty graph is an independent
 set, 
Theorem~\ref{theorem_chromatic}
yields the classical formula
\begin{equation}
S(n,k)
=
\frac{1}{k!} \cdot \Delta^{k}(x^{n})\rvert_{x=0}
=
\frac{1}{k!} \cdot \sum_{i=0}^{k} (-1)^{k-i} \cdot \binom{k}{i} \cdot i^{\,n}.
\label{equation_chromatic_classic_stirling}
\end{equation}

Furthermore, since the chromatic polynomial of a cycle of length $n \geq 3$ is given
by $\chi(C_{n};x) = (x-1)^{n} + (-1)^{n} \cdot (x-1)$, we have the following
consequence.
\begin{proposition}
The number of set partitions of $[n]$
into $k \geq 2$ blocks
such that the elements $i$ and $i+1$ are in different blocks, including $1$ and $n$,  
is given by
the box polynomial $B_{n-k,k}(x)$ evaluated at $x=-1$.
\label{proposition_chromatic_acyclic}
\end{proposition}
\begin{proof}
By Theorem~\ref{theorem_chromatic}
the sought after enumeration is given by:
\begin{align*}
\frac{1}{k!}
\cdot
\Delta^{k}\left(
(x-1)^{n} + (-1)^{n} \cdot (x-1)
\right)
\rvert_{x=0}
& =
\frac{1}{k!}
\cdot
\Delta^{k}\left( (x-1)^{n} \right) \rvert_{x=0}
=
\frac{1}{k!}
\cdot
\Delta^{k}\left( x^{n} \right) \rvert_{x=-1} ,
\end{align*}
which is the box polynomial $B_{n-k,k}(x)$
evaluated at $x=-1$ by Theorem~\ref{theorem_operator_interpretation}.
\end{proof}

A different inclusion-exclusion proof can be given for 
Proposition~\ref{proposition_chromatic_acyclic}.

\begin{proof}[Second proof of Proposition~\ref{proposition_chromatic_acyclic}]
Consider the set $[n]$ as the
congruence classes modulo $n$,
that is,~$\mathbb{Z}_{n}$.
In other words, the element $n$ is followed by $1$.
Let $A$ be a subset of $\mathbb{Z}_{n}$.
Then the number of set partitions $\pi$ of $[n]$ into $k$ blocks such that if $i$ belongs to $A$
then $i$ and $i+1$ belong to the same block of $\pi$ is given by
$S(n-|A|,k)$, since we can first choose a set partition
of $\mathbb{Z}_{n} - A$ and then insert $i \in A$ into the same block
as $i+1$.
Hence by inclusion-exclusion the desired number of set partitions
is given by
\begin{equation}
\sum_{A \subseteq \mathbb{Z}_{n}} (-1)^{|A|} \cdot S(n-|A|,k)
=
\sum_{j=0}^{n} (-1)^{j} \cdot \binom{n}{j} \cdot S(n-j,k) .
\label{equation_box_minus_one}
\end{equation}
Observe that when the variable $j$ exceeds $n-k$,
the associated term vanishes.
Now the result follows by
Proposition~\ref{proposition_closed_form_stirling_binomial}.
\end{proof}

We continue to apply Theorem~\ref{theorem_chromatic} to more
families of graphs. An {\em $s$-tree} is defined recursively as follows.
The complete graph $K_{s}$ is an $s$-tree.
Given an $s$-tree $T$ with a clique of size $s$,
then we can adjoin a new vertex only connected to all the
$s$ vertices in the clique, to obtain a new $s$-tree.
For instance, a $1$-tree is the classical notion of a tree.
Directly, we know that the chromatic polynomial of an $s$-tree
on $n$ vertices is given by
$x_{(s)} \cdot (x-s)^{n-s}$,
where $x_{(s)}$ denotes the lower factorial
$x \cdot (x-1) \cdots (x-s+1)$.
We now reproduce a result of Yang~\cite{yang}:

\begin{proposition}
Let $T$ be an $s$-tree on the vertex set $[n]$
and $k$ a positive integer such 
that $s \leq k \leq n$.  
The number of set partitions $\pi$ of $[n]$ into 
$k$ blocks such that if $i$ and $j$ 
are in the same block of $\pi$
then $i$ and $j$ are not adjacent in $T$
is given by the box polynomial
$B_{n-k,k-s}(x)$ evaluated at~$0$,
that is, the Stirling number
$S(n-s,k-s)$.
\label{proposition_s-tree}
\end{proposition}
\begin{proof}
By Lemma~\ref{lemma_Delta_x}
we have that
\begin{equation}
\Delta^{k}\underline{x} (p(x)) \rvert_{x=0}
=
k \cdot \Delta^{k-1}E (p(x)) \rvert_{x=0} .
\label{equation_delta_underline_x}
\end{equation}
Applying equation~\eqref{equation_delta_underline_x} to
the polynomial
$(x-1)_{(s-1)} \cdot (x-s)^{n-s}$ we have
\begin{align}
\frac{1}{k!} \cdot
\Delta^{k}
(x_{(s)} \cdot (x-s)^{n-s})
\rvert_{x=0}
& =
\frac{1}{(k-1)!} \cdot
\Delta^{k-1} E
((x-1)_{(s-1)} \cdot (x-s)^{n-s})
\rvert_{x=0}
\nonumber \\
& =
\frac{1}{(k-1)!} \cdot
\Delta^{k-1}
(x_{(s-1)} \cdot (x-s+1)^{n-s})
\rvert_{x=0}.
\label{equation_s}
\end{align}
By applying equation~\eqref{equation_s} $s-1$ more times, for a total of $s$ applications,
we obtain:
\begin{equation*}
\label{equation_chromatic_at_least_s_away}
\frac{1}{k!} \cdot \Delta^k(\chi(G;x))\rvert_{x=0}
=
\frac{1}{(k-s)!} \cdot \Delta^{k-s}(x^{n-s})\rvert_{x=0}
=
B_{n-k,k-s}(0) ,
\end{equation*}
which is the Stirling number $S(n-s,k-s)$.
\end{proof}

Furthermore by summing the number of blocks in
Proposition~\ref{proposition_s-tree}
between $s$ and $n$ we obtain:
\begin{corollary}
Let $T$ be an $s$-tree on the vertex set $[n]$. 
The number of set partitions $\pi$ of $[n]$
such that if $i$ and $j$ 
are in the same block of $\pi$
then $i$ and $j$ are not adjacent in $T$
is given by the Bell number $B(n-s)$.
\end{corollary}

The following result is originally due to to Prodinger~\cite{prodinger}.
See also Chen--Deng--Du for a bijective proof~\cite{chen_non_crossing}.
Their bijective proof can be described by
using the bijection $\varphi_{n}$ between set partitions
and rook placements on a triangular board;
see the description after Corollary~2.4.2
in~\cite{EC1}.
First apply this bijection $\varphi_{n}$
to obtain a rook placement on triangular board,
where the condition implies that there
is no rook on the $s$ longest diagonals,
then remove these $s$ diagonals and apply the
inverse map~$\varphi_{n-s}^{-1}$ to obtain
a set partition on the set $[n-s]$.
\begin{corollary}
Let $k$ and $s$ be positive integers such 
that $k\geq s$.  
The number of set partitions $\pi$ of $[n]$ into 
$k$ blocks such that if $i$ and $j$ 
are in the same block of $\pi$
then $|i-j|>s$ is given by the Stirling number $S(n-k,k-s)$.
\label{corollary_at_least_s_apart}
\end{corollary}
\begin{proof}
It is enough to observe that
the graph $G$ on the vertex set $[n]$ such that 
$i$ and $j$ are adjacent if $|i-j|\leq s$
is an $s$-tree.
\end{proof}

\begin{proposition}
Let $r$ be a positive integer.
The box polynomial~$B_{m,n}(x)$ evaluated at $x=r$ enumerates
set partitions of $m+n+r$ elements into $n+r$ blocks
such that the elements $1,2,\ldots,r$ are all in different blocks.
\label{proposition_1_r_different_blocks_chromatic}
\end{proposition}
\begin{proof}
Apply Theorem~\ref{theorem_chromatic}
to the graph $G$ given by the disjoint union of
the complete graph $K_{r}$ and $m+n$ isolated
vertices. The chromatic polynomial of $G$
is
$\chi(G;x) = x_{(r)} \cdot x^{m+n}$.
By applying equation~\eqref{equation_delta_underline_x}
$r$ times to the expression
$\frac{1}{(n+r)!} \cdot \Delta^{n+r}(\chi(G;x))\rvert_{x=0}$
we are left with:
\begin{align*}
\frac{1}{(n+r)!} \cdot \Delta^{n+r}(\chi(G;x))\rvert_{x=0}
& =
\frac{1}{n!} \cdot \Delta^{n}((x+r)^{m+n})\rvert_{x=0}
=
\frac{1}{n!} \cdot \Delta^{n}(x^{m+n})\rvert_{x=r}
=
B_{m,n}(r) .
\qedhere
\end{align*}
\end{proof}

\section{Connection with set partitions via generating functions}
\label{section_box_polynomial_generating_functions}

We now turn our attention to generating functions and their connection to
enumeration of set partitions.
Our tool are restricted growth words.

A {\em restricted growth word}, or $\RG$-word for short,
is a word $v = v_{1} v_{2} \cdots v_{n}$ whose letters
are positive integers such that
$v_{i} \leq \max(0, v_{1}, v_{2}, \ldots, v_{i-1}) + 1$.
There is a natural bijection between set partitions
of the set $[n]$ into $k$ blocks and $\RG$-words of length $n$
such that the largest letter is~$k$. Namely, if $v_{i} = v_{j}$
place $i$ and $j$ in the same block.
The inverse of this bijection is given by ordering the blocks
of the partition $\pi = \{B_{1}, B_{2}, \ldots, B_{k}\}$
such that $\min(B_{1}) < \min(B_{1}) < \cdots < \min(B_{k})$
and then letting $v_{i} = r$ if $i$ belongs to the $r$th block $B_{r}$.
Furthermore, a partition with the blocks ordered according to their
smallest elements is said to be in {\em standard form}.

Note that every $\RG$-word $v$ has a unique factorization
as $v = 1 \cdot u_{1} \cdot 2 \cdot u_{2} \cdot 3 \cdots k \cdot u_{k}$,
where the word~$u_{i}$ only has letters from the interval $[i]$.
Since the length generating function for words with letters from
an alphabet of size $i$ is $1/(1 - i \cdot t)$ we directly obtain
the generating function
for the Stirling numbers of the second kind
\begin{align}
\sum_{n \geq k} S(n,k) \cdot t^{n-k}
=
\frac{1}{1-t}
\cdot
\frac{1}{1 - 2 \cdot t}
\cdots
\frac{1}{1 - k \cdot t} .
\label{equation_restriction_idea}
\end{align}
However, if we have restrictions on the word $u_{i}$
this will yield a different $i$th factor in the
product in equation~\eqref{equation_restriction_idea}.
This methodology is used three times in this section.

We now give yet another proof of 
Proposition~\ref{proposition_chromatic_acyclic} using 
restricted growth words.

\begin{proof}[Third proof of Proposition~\ref{proposition_chromatic_acyclic}]
The associated $\RG$-word for such a set partition factors as
$1 \cdot u_{1} \cdot 2 \cdot u_{2} \cdot 3 \cdots k \cdot u_{k}$
where the word $u_{i}$ does not begin with letter $i$
for $1 \leq i \leq k$
and the word $u_{k}$ does not end
in the letter $1$.
These condition imply that
$u_{1}$ is the empty word
and hence its associated length generating function is $1$.
For $2 \leq i \leq k-1$
the word $u_{i}$ is either empty
or it begins
with a letter $1$ through $i-1$. 
For each of the subsequent letters
of $u_i$, we need the next letter different from
the previous letter, yielding $i-1$ choices for letter.
Therefore
the length generating function for $u_{i}$,
where $i \leq k-1$,
is
\begin{align}
1 + (i-1)t \cdot \frac{1}{1 - (i-1) \cdot t} = \frac{1}{1 - (i-1) \cdot t} .
\label{equation_u_i}
\end{align}

The argument for the length generating function for the last
word $u_{k}$ is more delicate since we have restrictions on both
the first and last letter.
Let $M$ be the $k \times k$ matrix where all the entries are $1$
but the diagonal entries which are $0$.
Observe that the $(i,j)$ entry of $M^{n-1}$ enumerates the 
number of words of length $n$ such that the first letter is $i$, the last letter is $j$
and each pair of adjacent letters are different.
Let $\vec{a}$ be the vector $(0, 1, 1, \ldots, 1)$
and
$\vec{b}$ be the vector $(1, 1, \ldots, 1, 0)^{T}$.
Then the number of words $u_{k}$ of length $n \geq 1$
is given by the product
$\vec{a} \cdot M^{n-1} \cdot \vec{b}$.
Observe that the vector
$\vec{v}^{\,T} = (1,1, \ldots, 1)$
is an eigenvector of $M$ with eigenvalue $k-1$.
Similarly,
$\vec{w}^{\,T} = (-k+1, 1, 1, \ldots, 1)$
is an eigenvector with eigenvalue $-1$.
Since
$\vec{b} = (1-1/k) \cdot \vec{v} + 1/k \cdot \vec{w}$
we have
\begin{align*}
\vec{a} \cdot M^{n-1} \cdot \vec{b}
& =
\vec{a} \cdot M^{n-1} \cdot ((1-1/k) \cdot \vec{v} + 1/k \cdot \vec{w}) \\
& =
\vec{a} \cdot
((1-1/k) \cdot (k-1)^{n-1} \cdot \vec{v} + 1/k \cdot (-1)^{n-1} \cdot \vec{w})
\\
& =
(1-1/k) \cdot (k-1)^{n} + 1/k \cdot (-1)^{n} .
\end{align*}
Note that this expression also holds when $n=0$
enumerating the empty word.
It remains to observe that
the length generating function is given by
\begin{align}
\sum_{n \geq 0}
\left(\left(1-\frac{1}{k}\right) \cdot (k-1)^{n} + \frac{(-1)^{n}}{k} \right) \cdot t^{n}
 & =
 \frac{1-1/k}{1 - (k-1) \cdot t} + \frac{1/k}{1+t}
 \nonumber \\
 & =
\frac{1}{(1 - (k-1) \cdot t) \cdot (1+t)}.
\label{equation_eigenvalue_generating_function}
\end{align}
Therefore, multiplying
equations~\eqref{equation_u_i} for $2 \leq i \leq k-1$
and equation~\eqref{equation_eigenvalue_generating_function}, the 
generating function for these $\RG$-words
is given by
\[
\frac{t^{k}}
{(1+t) \cdot (1-t) \cdot (1-2t) \cdots (1-(k-2) \cdot t) \cdot
(1-(k-1) \cdot t)} ,
\]
which is the generating function
$\sum_{n \geq k} B_{n-k,k}(-1) \cdot t^{n}$.
\end{proof}

The generating function appearing in this proof has been studied before: see the Monthly problem by Knuth and solved by
Lossers~\cite{knuth_monthly}. They present a different derivation of this generating function.

\begin{proof}[Second proof of Proposition~\ref{proposition_1_r_different_blocks_chromatic}]
Set $x = r$ in the generating function
in equation~\eqref{equation_box_generating_function}
to obtain
\begin{align}
\sum_{m \geq 0} B_{m,n}(r) \cdot t^{m}
& =
\frac{1}{(1 - r \cdot t) \cdot (1 - (r+1) \cdot t) \cdots (1 - (r+n) \cdot t)} .
\label{equation_x_r}
\end{align}
Observe that this is generating function of the number
of restricted growth words of the form
\begin{align*}
w = 1 \cdot 2 \cdots r \cdot u_{r} \cdot (r+1)\cdot u_{r+1} \cdots (r+n) \cdot u_{r+n} ,
\end{align*}
where $u_{i}$ is a word in the letters $1$ through $i$
and the words $u_{1}$ through $u_{r-1}$ are empty.
These restricted growth words
are in direct bijection with 
set partitions such 
that the elements $1, 2, \ldots, r$ all belong to separate blocks.
\end{proof}

\begin{proposition}
Let $r$ be a positive integer.
Then the box polynomial evaluated at $x=r$, $B_{m,n}(r)$,
is given by the sum
\[
B_{m,n}(r)=\sum_{i=0}^{r-1}
s(r,r-i) \cdot S(m+n+r-i, r+n) ,
\]
where $s(r,i)$ denotes the (signed) Stirling number of the first kind.
\end{proposition}
\begin{proof}
By equation~\eqref{equation_x_r}
we have that
\begin{align*}
\sum_{m \geq 0} B_{m,n}(r) \cdot t^{m}
& =
\frac{(1-t) \cdot (1 - 2 \cdot t) \cdots (1 - (r-1) \cdot t)}
{(1 - t) \cdot (1 - 2 \cdot t) \cdots (1 - (r+n) \cdot t)} \\
& =
\left( \sum_{i=0}^{r-1} s(r,r-i) \cdot t^{i} \right)
\cdot
\left(
\sum_{j \geq 0} S(j+n+r,n+r) \cdot t^{j}
\right) .
\end{align*}
We used that
$p(t)= t \cdot (t-1) \cdots (t-(r-1)) = \sum_{k=0}^{r} s(r,k) \cdot t^{k}$, see~\cite[Proposition~1.3.7]{EC1}. 
The coefficient of $t^{m}$
follows by multiplying these two generating functions.
\end{proof}

\begin{proposition}
Let $r$ be a positive integer such that $n \geq 2r$. The box polynomial $B_{m,n}(x)$ 
evaluated at the integer $-r$ enumerates set partitions $\pi$
of $m+n-r$ elements into $n-r$ blocks
such that the minimal element of the $i$th block $B_{i}$
in the standard form of $\pi$
is congruent to $i$ modulo $2$ for $1\leq i\leq r+1$
when $r \neq n/2$. 
When $r = n/2$, the range of $i$ should be $1 \leq i \leq r$,
since there is no $(r+1)$st block.
\label{proposition_box_at_minus_r}
\end{proposition}
\begin{proof}
Set $x=-r$ in equation~\eqref{equation_box_generating_function} to obtain:
\begin{align*}
\sum_{m \geq 0} B_{m,n}(-r) \cdot t^{m}
& =
\frac{1}{(1 + r \cdot t) \cdot (1 + (r-1) \cdot t) \cdots (1 +t)\cdot(1-t)\cdots(1-(n-r)\cdot t) } .
\end{align*}
Since $n-r\geq r$,
we can pair factors in this denominator to obtain the expression:
\begin{equation*}
\frac{1}{(1-t^2)\cdot(1-(2t)^2)\cdots(1-(rt)^2)\cdot(1-(r+1)\cdot t)\cdots(1-(n-r)\cdot t)}.
\end{equation*}
This expression
is the generating function for restricted growth 
words of the form
\[
w = 1 \cdot u_{1} \cdot 2 \cdot u_{2} \cdots r \cdot u_{r} \cdot (r+1) \cdot u_{r+1} \cdots (n-r) \cdot u_{n-r},
\]
such that the length of the word $u_{i}$ is even for $1 \leq i\leq r$. 
This implies that the length of $1 \cdot u_{1} \cdot 2 \cdot u_{2} \cdots i \cdot u_{i}$ 
has the same parity as $i$. In other words, the minimal element of the $(i+1)$st block of 
the partition~$\pi$ (in standard form) has the same parity as $i+1$ for $1\leq i\leq r$.
\end{proof}

\begin{proposition}\label{proposition_odd_set_partitions}
The expression
$2^{m} \cdot B_{m,n}(-n/2)$
enumerates
set partitions 
of a set of cardinality $m+n$ into $n$ blocks of odd size,
denoted by $T_{m+n,n}$.
\end{proposition}
\begin{proof}
Using equation~\eqref{equation_box_via_Delta}
evaluated at $x=-n/2$ yields
\begin{equation}
2^{m} \cdot B_{m,n}(-n/2)
=
\frac{1}{2^{n} \cdot n!}
\cdot
\sum_{r=0}^{n}
(-1)^{n-r} \cdot
\binom{n}{r} \cdot
(2r-n)^{m+n} .
\label{equation_odd}
\end{equation}
The exponential generating
function for partitions with
$n$ blocks with odd cardinalities
is
$\sinh(x)^{n}/n!$ $=$ $(e^{x}-e^{-x})^{n}/(2^{n} \cdot n!)$.
Using the binomial theorem and
considering the coefficient of
$x^{m+n}/(m+n)!$ yields the
right hand side
of equation~\eqref{equation_odd}.
\end{proof}

\begin{remark}
\label{remark_box_roots}
From Lemma~\ref{lemma_even_odd}
it follows that when $m$ is odd the box polynomial $B_{m,n}$
has $-n/2$ as a root of odd multiplicity.
This also follows from
Proposition~\ref{proposition_odd_set_partitions}
since when $m$ is odd $m+n$ and $n$ have different parities.
However, when $m$ is even and greater than or equal to $2$
there are at least 
$\binom{m+n}{n-1}$
partitions of $m+n$ into $n$ odd sized blocks,
namely the set partitions
consisting of $n-1$ singleton blocks and
one block of size $m+1$.
Therefore $-n/2$ is not a root of the box polynomial when $m \geq 2$ is even.
Finally, returning to the case when $m$ is odd, we know that the root $-n/2$
does not have multiplicity greater than $1$,
since by Lemma~\ref{lemma_D} this would imply 
that its derivative $B_{m-1,n}(-n/2)$ has a root at $-n/2$,
contradicting that $m-1$ is even.
\end{remark}

We now continue to discuss the number of set partitions
where all the blocks have odd cardinality. 
We begin to express this number in terms 
of Stirling numbers of the second kind.

\begin{corollary}
\label{corollary_odd_parts_generating_function}
For $n$ even, the ordinary generating function for
the numbers $T_{m+n,n}$ is given by
\[
\sum_{m \geq 0} T_{m+n,n} \cdot t^{m}
=
\frac{1}
{(1 - 2^{2} \cdot t^{2}) \cdot (1 - 4^{2} \cdot t^{2}) \cdots (1 - n^{2} \cdot t^{2})}.
\]
\label{corollary_t_m_n_generating_function}
\end{corollary}
\begin{proof}
By Proposition~\ref{proposition_odd_set_partitions}
and the generating function in equation~\eqref{equation_box_generating_function}
we have
that
\begin{align*}
\sum_{m \geq 0} T_{m+n,n} \cdot t^{m}
& =
\sum_{m \geq 0} B_{m,n}(-n/2) \cdot (2t)^{m} \\
& =
\frac{1}{(1 + n/2 \cdot 2t) \cdot (1 + (n/2 -1) \cdot 2t) \cdots 
(1 - (n/2 -1) \cdot 2t) \cdot (1 - n/2 \cdot 2t)} .
\end{align*}
The last step is combine factors
using $(1 + k \cdot t) \cdot (1 - k \cdot t) = 1 - k^{2} \cdot t^{2}$.
\end{proof}

By equating the coefficients of $t^{m}$ in
Corollary~\ref{corollary_t_m_n_generating_function}
we have an immediate corollary.
\begin{corollary}
\label{corollary_T_m_n_homogeneous}
For $m$ and $n$ both even, 
the number $T_{m+n,n}$ is given by 
the complete symmetric function
\[ T_{m+n,n}=h_{m/2}(2^{2},4^{2},\ldots,n^{2}) . \]
\end{corollary}
See also the proof of this result in~\cite{RG_words} using $\RG$-words,
multivariate generating functions and integer walks.

We now look at another consequence of the generating function for $B_{m,n}(x)$ given in
equation~\eqref{equation_box_generating_function}.
\begin{corollary}
\label{corollary_T_m_n_stirling_numbers}
Let $n$ be an even integer.
Then the number of set partitions of a set of cardinality $m+n$
into $n$ blocks of odd size is given by
the following convolution of Stirling numbers of the second kind
\[
T_{m+n,n}
=
2^{m} \cdot
\sum_{k=0}^{m} (-1)^{k} \cdot S(k+n/2,n/2) \cdot  S(m-k+n/2,n/2).
\]
\end{corollary}
\begin{proof}
We factor the generating function for $2^{-m} \cdot T_{m+n,n}$
as 
\begin{align*}
\sum_{m \geq 0} T_{m+n,n} \cdot (t/2)^{m}
& =
\frac{1}{(1 - t^{2}) \cdot (1 - 2^{2} \cdot t^{2}) \cdots (1 - (n/2)^{2} \cdot t^{2})} \\
& =
\frac{1}{(1 + t) \cdot (1 + 2 \cdot t) \cdots (1 + n/2 \cdot t)}
\cdot
\frac{1}{(1 - t) \cdot (1 - 2 \cdot t) \cdots (1 - n/2 \cdot t)} .
\end{align*}
The second factor is the generating function for
the Stirling numbers $S(m+n/2,n/2)$.
The first factor is the generating function for
$(-1)^{m} \cdot S(m+n/2,n/2)$.
The result follows since the product of generating functions
corresponds to the convolution of the coefficients.
\end{proof}

A second proof of Corollary~\ref{corollary_T_m_n_stirling_numbers}
comes via
Proposition~\ref{proposition_odd_set_partitions}
and 
Lemma~\ref{lemma_split} with $n = (n/2 - 1) + n/2 + 1$,
\begin{align*}
T_{m+n,n}
&=
2^{m} \cdot
B_{m,n}(-n/2) \\
&=
2^{m} \cdot
\sum_{k=0}^{m} B_{k,n/2-1}(-n/2) \cdot B_{m-k,n/2}(0) \\
&=
2^{m} \cdot
\sum_{k=0}^{m} (-1)^{k} \cdot S(k+n/2,n/2) \cdot  S(m-k+n/2,n/2) ,
\end{align*}
where the last step uses
equations~\eqref{equation_box_polynomial_at_zero}
and~\eqref{equation_box_polynomial_at_minus_n_minus_one}.

\section{Bijective interlude}
\label{section_bijective_interlude}

A few results in the previous two sections beg for bijective proofs.
We first extend the notion~$S(G,k)$ to collections of partitions.
For a graph $G$
let~$\mathcal{S}(G,k)$ be the collection of set partitions of the vertex set of $G$
into $k$ blocks such that adjacent vertices of $G$ 
are in different blocks.
Hence~$S(G,k)$ is the cardinality of $\mathcal{S}(G,k)$,
that is, $S(G,k)=|\mathcal{S}(G,k)|$.

The chromatic polynomial of a tree $T$ on $n$ vertices is $t \cdot (t-1)^{n-1}$.
Hence by Theorem~\ref{theorem_chromatic}, the number of
partitions in $\mathcal{S}(T,k)$ is independent of
the tree $T$. We give a bijective proof of this fact.
Recall that a thicket is a forest with two connected
components; see~\cite{bollobas}.

\begin{proposition}
\label{proposition_tree_bijection}
Let $F = H_{1} \cup H_{2}$
be a thicket comprised of
the trees $H_{1}$ and $H_{2}$. Let $x$ and $z$ be
(not necessarily distinct) vertices belonging to
$H_{1}$
and $y$ and $w$ be (not necessarily distinct) vertices belonging to $H_{2}$,
such that $T_{1} = F \cup \{xy\}$ and
$T_{2} = F \cup \{zw\}$ are both trees.
Then there is a bijection between
$\mathcal{S}(T_{1},k)$ and $\mathcal{S}(T_{2},k)$.
\end{proposition}
\begin{proof}
Given a partition $\pi$ in $\mathcal{S}(T_{1},k)$.
Let $B_{1}$, $B_{2}$, $C_{1}$ and $C_{2}$ be the four blocks of $\pi$
containing the elements $x$, respectively, $y$, $z$ and $w$.
For $i = 1,2$ switch elements between blocks $B_{i}$ and $C_{i}$ if
they belong to the subtree $H_{i}$.
The partition $\tau$ created this way is a partition whose
blocks are independent sets of the tree $T_{2}$.
Furthermore, the map $\pi \longmapsto \tau$
is a bijection that preserves the number of blocks~$k$.
\end{proof}

Since any tree can be obtained from any other tree by switching edges we obtain
the desired bijection.
As an example, consider the star tree,
that is, the tree where every edge is connected to a given vertex, say $n$.
For this tree, the cardinality of $\mathcal{S}(T,k)$ is given by
the Stirling number $S(n-1,k-1)$ of the second kind, since the vertex $n$
must be a singleton block.
Hence we know that for any tree $T$ on $n$ vertices the cardinality of
$\mathcal{S}(T,k)$ is $S(n-1,k-1)$.

It is a known fact that the set partitions of
an $n$-set without singleton blocks is equinumerous with the
set partitions of the same set without two cyclically
consecutive elements
in the same block. A bijective
proof was given by Callan~\cite[Theorem~1]{Callan}.
Here we present a shorter bijection.
\begin{lemma}
There is a bijection between
set partitions of $[n]$ such that
$i$ and $i+1$ are not in the same block,
including $1$ and $n$,
and
set partitions of $[n]$ with no singleton blocks.
\label{Stanley_exercise}
\end{lemma}
\begin{proof}
Define a map
$\varphi : \bigcup_{k=0}^{n} \mathcal{S}(C_{n},k)
\longrightarrow \mathcal{Q}_{n}$ by the following
procedure,
where
$\mathcal{Q}_{n}$ denotes the collection
of set partitions of $[n]$ containing no
singleton blocks.
If $\pi$ consists of all singletons,
then assign $\varphi(\pi)$ to the partition $\{[n]\}$,
that is, the partition consisting of a single block.
Otherwise, consider each maximal run
of singleton blocks
$\{i\}, \{i+1\}, \ldots, \{i+j\}$
where a run is defined modulo~$n$,
and $B$ is the non-singleton block
containing the element $i+j+1$.
Merge the singletons $\{i\}$ and $\{i+1\}$,
the singletons $\{i+2\}$ and $\{i+3\}$,
and so on.
If $j$ is odd, the last pair to be merged
together is
$\{i+j-1\}$ and $\{i+j\}$.
If $j$ is even, the last pair to be merged
together is
$\{i+j\}$ and the non-singleton block $B$.
Note that $\varphi$ maps into $\mathcal{Q}_{n}$.

Note, however, that
$\varphi$ is not necessarily a bijection.
When $n$ is even,
note that
the two partitions
$$
\pi_{1} = 13\cdots(n-1)|2|4| \cdots |n
\,\text{ and }\,
\pi_{2} = 1|3|\cdots|(n-1)|24 \cdots n
$$
both map to the partition $\{[n]\}$.
Hence, define the map
$\psi$ by
$\psi(\pi_{1}) = 23|45|\cdots|n1$,
$\psi(\pi_{2}) = 12|34|\cdots|(n-1)n$
and
$\psi(\pi) = \varphi(\pi)$
for $\pi \neq \pi_{1},\pi_{2}$.

The reverse map of $\psi$
is defined by first considering the three special cases.
The merging process is reversed by starting
with any block containing $i$ and $i+1$, but not $i+2$.
Then remove $i$ to its own block, and starting with $i-1$,
look for the next smaller occurrence of an adjacency
and continue to perform this operation.
It is straightforward to see that these maps
are inverses of one another.
\end{proof}

\begin{corollary}
For $n \geq 2$,
the number of set partitions of $[n]$
with no singleton blocks is given by
$$
\sum_{j=2}^{n} B_{n-j,j}(-1) .
$$
\label{corollary_no_singleton}
\end{corollary}
\begin{proof}
Using
Proposition~\ref{proposition_chromatic_acyclic},
the box polynomial
$B_{n-j,j}(-1)$ counts set partitions of $n$ into $j$ blocks 
avoiding $i$ and $i+1$ in the same block cyclically. The result now 
follows by summing over all block sizes and with Lemma~\ref{Stanley_exercise}.
\end{proof}

\begin{remark}
{\rm
Proposition~\ref{proposition_chromatic_acyclic}
and~\ref{proposition_box_at_minus_r}
both yield a combinatorial interpretation
for the box polynomial evaluated at $x=-1$,
that is, $B_{n-k,k}(-1)$.
First, we have $S(C_{n},k)$,
which enumerates set partitions of $[n]$ into $k$ blocks
such that
$i$ and $i+1$ do not belong to same block
and $1$ and $n$ also do not belong to same block.
Second, we have
set partitions
of $[n-1]$ into $k-1$ blocks written in standard form such that
the minimal element of the second block
is even. A bijection proving that these
two interpretations are equinumerous
is given by
restricting
the bijection in the solution to
Exercise~108(a) in~\cite[Chapter 1]{EC1}.
}
\end{remark}

\section{Fast Fourier operators}
\label{section_fast_fourier}

In this section, we generalize Proposition~\ref{proposition_odd_set_partitions} to set partitions with block sizes $1$ modulo $r$.
Let $\omega=e^{2\pi i/r}$ be a primitive $r$th root of unity.
Recall that $E^{a}$ is the shift operator $E^{a}(p(x)) = p(x+a)$.
Furthermore, an operator $T$ is
shift invariant if it commutes with $E^{a}$ for all $a$.
\begin{definition}
\label{fast_fourier_definition}
The \emph{fast Fourier operator} $F_{r}$ is given by
$$
F_{r}
=
\frac{1}{r} \cdot \sum_{j=0}^{r-1}\omega^{-j} \cdot E^{\omega^{j}} . 
$$
\end{definition}

Since $\omega$ is a complex number, it is not clear that $F_{r}$ applied to a real polynomial $p(x)$ is still a real polynomial.
Let $f_{r}(x)$ be the generating function
$\sum_{m \equiv 1 \bmod r} {x^{m}}/{m!}$.
\begin{lemma}
The fast Fourier operator $F_{r}$ is
given by $f_{r}(D)$, where $D$ is the derivative operator.
Especially, the fast Fourier operator restricts to
an operator on the polynomial ring
$\mathbb{R}[x]$.
\end{lemma}
\begin{proof}
By Taylor's theorem we have that
$E^{a}=\sum_{m \geq 0}(aD)^{m}/m!$;
see~\cite[Theorem~2]{rota}, hence 
\begin{align*}
F_{r}
& =
\frac{1}{r} \cdot \sum_{j=0}^{r-1} \omega^{-j} \cdot \sum_{m \geq 0} \omega^{j \cdot m} \cdot \frac{D^{m}}{m!}
=
\frac{1}{r} \cdot \sum_{m \geq 0} \frac{D^{m}}{m!} \cdot \sum_{j=0}^{r-1} \omega^{j \cdot (m-1)}
=
\sum_{m \equiv 1 \bmod r}\frac{D^{m}}{m!}.
\qedhere
\end{align*}
\end{proof}

The name \emph{fast Fourier operator} comes from the system of equations one must solve to find the
complex coefficients $\alpha_{j}$ such that
$f_{r}(x)
=
\sum_{j=0}^{r-1} \alpha_{j} \cdot e^{\omega^{j} \cdot x}$.
In particular, to arrive at
Definition~\ref{fast_fourier_definition}, 
one needs to invert the
$r$ by $r$ matrix whose $(i,j)$ entry
is $\omega^{i \cdot j}$,
that is,
the \emph{fast Fourier matrix}.

The next result shows that the fast Fourier operator $F_{r}$
satisfies an analogous result to Proposition~\ref{proposition_odd_set_partitions}.

\begin{proposition}
\label{fast_fourier_set_partitions}
The number of partitions of the set $[m]$ into $k$ blocks of cardinality $1$ modulo $r$ is given by
$1/k! \cdot F_{r}^{k} (x^{n})\big|_{x=0}$.
\end{proposition}
\begin{proof}
We have
\begin{align*}
F_{r}^{k} (x^{n})\big|_{x=0}
& =
\left.
\left(
\sum_{b_{1} \equiv 1 \bmod r} \frac{D^{b_{1}}}{b_{1}!}
\right)
\cdots
\left(
\sum_{b_{k} \equiv 1 \bmod r} \frac{D^{b_{k}}}{b_{k}!}
\right)
(x^{n})\right|_{x=0} 
= 
\sum_{\substack{b_{1} + \cdots + b_{k} = n \\ b_{j} \equiv 1 \bmod r}}
\frac{n!}{b_{1}! \cdots b_{k}!} ,
\end{align*}
which is the number of ordered set partitions of $[n]$ into $k$ blocks
with each block size is congruent to $1$ modulo $r$.
Dividing by $k!$ yields the result.
\end{proof}

A \emph{delta operator} is a polynomial operator $T$ such that $T$ is shift invariant
and $T(x)\neq 0$; see~\cite{rota}.
The fast Fourier operator $F_{r}$ is shift invariant, as it is a linear combination of shift operators,
and
$F_{r}(x) = \sum_{j=0}^{r-1} \omega^{-j} \cdot (x+\omega^{j}) = r$, hence $F_{r}$ is a delta operator.

Associated to any delta operator $T$ is a \emph{basic sequence} of polynomials $p_{n}(x)$ such that:
(i)
$T(p_{n}(x)) = n \cdot p_{n-1}(x)$,
(ii)
$p_{0}(x)=1$,
(iii)
$p_{n}(0)=0$ for $n>0$. 
The sequence of basic polynomials for any delta operator $T$ is a sequence of binomial type,
that is,
$p_{n}(x+y) = \sum_{k=0}^{n} \binom{n}{k} \cdot p_{k}(x) \cdot p_{n-k}(y)$.
We now determine the basic sequence for the fast Fourier operator $F_{r}$.

Recall that $\Pi_{n}$ is the partition lattice on a set of cardinality $n$.
Define $\Pi_{n}^{r,1}$ to be the subposet of $\Pi_{n}$
consisting of all set partitions where the block sizes are
congruent to $1$ modulo $r$, that is,
$$
\Pi_{n}^{r,1}
=
\{\pi \in \Pi_{n} 
\:\: : \:\: 
\forall B \in \pi \: |B| \equiv 1 \bmod r
\} .
$$
Observe that $\Pi_{n}^{r,1}$ has a minimal element, that is, the
partition consisting of all singleton blocks.
When $n$ is congruent to $1$ modulo $r$,
$\Pi_{n}^{r,1}$ has a maximal element, namely the set partition
consisting of one block.
Finally, for $n$ congruent to $1$ modulo $r$ define
$\mu(n)$ to the M\"obius function $\mu(\Pi_{n}^{r,1}) = \mu_{\Pi_{n}^{r,1}}(\widehat{0},\widehat{1})$
and set $\mu(n)$ to be $0$ for $n \not\equiv 1 \bmod r$.

\begin{lemma}
The compositional inverse of the generating function
$f_{r}(x) = \sum_{n \equiv 1 \bmod r} x^{n}/n!$ is given by
$$
h_{r}(x)
=
\sum_{n \equiv 1 \bmod r}
\mu(n)
\cdot
\frac{x^{n}}{n!} .
$$
\end{lemma}
\begin{proof}
By composition of exponential generating functions we have that
\begin{align*}
\left[\frac{x^{n}}{n!}\right]
f_{r}(h_{r}(x))
& =
\sum_{\substack{\pi \in \Pi_{n} \\ |\pi| \equiv 1 \bmod r}}
\prod_{B \in \pi} \mu(|B|) 
=
\sum_{\substack{\pi \in \Pi_{n} \\ |\pi| \equiv 1 \bmod r}}
\prod_{B \in \pi} \mu(\Pi_{|B|}^{r,1})  .
\end{align*}
Observe that if $n \not\equiv 1 \bmod r$ then this sum is empty
and hence equal to $0$. Hence we continue under the
assumption that $n \equiv 1 \bmod r$.
\begin{align*}
\left[\frac{x^{n}}{n!}\right]
f_{r}(h_{r}(x))
& =
\sum_{\pi \in \Pi_{n}^{r,1}} \mu_{\Pi_{n}^{r,1}}(\widehat{0}, \pi)  
=
\delta_{n,1} .
\end{align*}
Hence the composition $f_{r}(h_{r}(x))$ is $x$, proving the lemma.
\end{proof}

\begin{theorem}
The sequence of basic polynomials for the fast Fourier operator $F_{r}$ is given by 
\[
p_{n}(x) = \sum_{\pi \in \Pi_{n}^{r,1}}  \prod_{B \in \pi} \mu(|B|)  \cdot x^{|\pi|} .
\] 
\end{theorem}
\begin{proof}
Using~\cite[Corollary~3]{rota}
and composition of exponential generating functions
we have
\begin{align*}
p_{n}(x)
& =
\left[ \frac{u^{n}}{n!} \right]
\sum_{j \geq 0} p_{j}(x) \cdot \frac{u^{j}}{j!}
=
\left[ \frac{u^{n}}{n!} \right]
e^{x \cdot h_{r}(u)}
=
\sum_{\pi \in \Pi_{n}^{r,1}}  \prod_{B \in \pi} (\mu(|B|) \cdot x) .
\qedhere
\end{align*}
\end{proof}

\begin{example}
The fast Fourier operator $F_{2}$ is related to the forward difference operator $\Delta$.
Let $M_{a}$ be the operator defined by the substitution $M_{a}(p(x)) = p(a \cdot x)$.
Then we have
$$
F_{2} = \frac{1}{2} \cdot M_{1/2} E^{-1/2} \Delta M_{2} . 
$$
Thus the $n$th power is given by
$$
F_{2}^{n} = 2^{-n} \cdot M_{1/2} E^{-n/2} \Delta^{n} M_{2} ,
$$
and we obtain
\begin{align*}
F_{2}^{n}(x^{m+n})
& =
2^{-n} \cdot M_{1/2} E^{-n/2} \Delta^{n} M_{2}(x^{m+n}) \\
& =
2^{m} \cdot M_{1/2} E^{-n/2} \Delta^{n} (x^{m+n}) \\
& =
2^{m} \cdot n! \cdot M_{1/2} E^{-n/2} B_{m,n}(x) \\
& =
2^{m} \cdot n! \cdot B_{m,n}\left((x-n)/{2}\right) ,
\end{align*}
showing that the polynomial $F_2^{n}(x^{m+n})$
is an affine transformation of the box polynomial $B_{m,n}(x)$.

\end{example}

\section{Bounds on the roots}
\label{section_box_polynomial_bounds_on_roots}

We now discuss the location of
the roots of the box polynomial $B_{m,n}(x)$.

\begin{theorem}
All roots of the box polynomial $B_{m,n}(x)$ have real part $-n/2$.
\label{theorem_root_real_part}
\end{theorem}
\begin{proof}
If the polynomial $p(x)$
has roots all with real part~$a$,
then the polynomial $\Delta(p(x))$
has roots with all real parts $a-1/2$.
This statement is due to
P{\'o}lya~\cite{Polya},
who stated it as a problem
which was solved by
Obreschkoff~\cite{Obreschkoff}.
(For a more general statement,
see~Lemma 9.13
in~\cite{Postnikov_Stanley}.)
Applying this result $n$ times to
the polynomial $x^{m+n}$
yields the result;
see~Theorem~\ref{theorem_operator_interpretation}.
\end{proof}

We can now improve Remark~\ref{remark_box_roots}.

\begin{corollary}
For $n \geq 1$, the box polynomial $B_{m,n}(x)$
has no multiple roots.
\end{corollary}
\begin{proof}
Note that the greatest common divisor satisfies
\begin{align*}
\gcd\left( B_{m,n}(x), \frac{d}{dx} B_{m,n}(x) \right)
& =
\gcd\left( x \cdot B_{m-1,n}(x) + B_{m,n-1}(x+1),
(m+n) \cdot B_{m-1,n}(x) \right) \\
& =
\gcd\left( B_{m,n-1}(x+1), B_{m-1,n}(x) \right) .
\end{align*}
Further, notice that all the roots
of $B_{m,n-1}(x+1)$
have real part $-(n-1)/2 - 1 = -n/2 - 1/2$,
whereas
$B_{m-1,n}(x)$
has all roots with real part $-n/2$.
Thus the greatest common divisor is $1$,
and hence, we conclude $B_{m,n}(x)$ has no multiple roots.
\end{proof}

\begin{example}
When $n=1$ the roots of the box polynomial 
$B_{m,1}(x) = (x+1)^{m+1} - x^{m+1}$
are given by
$$
- \frac{1}{2}
+
i \cdot \frac{1}{2}
\cdot
\frac{\sin\left(\frac{2\pi \cdot j}{m+1}\right)}
{\cos\left(\frac{2\pi \cdot j}{m+1}\right) -1}
$$
for $1 \leq j \leq m$.
Note that the largest imaginary part is about 
$(m+1)/2\pi$.
\end{example}

\begin{example}
When $n=2$ the real roots of the box polynomial 
$B_{m,2}(x)$ are of the form $-1 + i \cdot u$
where $u = \frac{v}{\sqrt{1-v^{2}}}$ and
$v$ is a root of the equation
$T_{m+2}(v) = v^{m+2}$,
where $T_{m+2}$ is the Chebyshev polynomial of the first kind.
\end{example}

\begin{example}
For $1 \leq m \leq 5$ the imaginary part of the roots of the
box polynomial $B_{m,n}(x)$ are listed in Table~\ref{table_box_imag_roots}.
\end{example}
\begin{table}
\begin{center}
\begin{tabular}{c | l}
$m$ & \\ \hline
$1$ & $0$ \\
$2$ & $\pm \sqrt{n/12}$ \\
$3$ & $0$, $\pm \sqrt{n/4}$ \\
$4$ & $\displaystyle \pm \sqrt{\frac{30 n\pm\sqrt{150n^2+30n}}{120}}$\\
$5$ & $\displaystyle 0, \pm \sqrt{\frac{10 n\pm\sqrt{5n^2+3n}}{24}}$
\end{tabular}
\end{center}
\caption{The imaginary parts of the roots of the box polynomial $B_{m,n}(x)$ for $1 \leq m \leq 5$.}
\label{table_box_imag_roots}
\end{table}
\begin{theorem}
The imaginary parts of the roots of the box polynomial $B_{m,n}(x)$
are bounded above by $mn/\pi$ and below by $-mn/\pi$.
\label{theorem_imaginary_bound}
\end{theorem}
\begin{proof}
Assume that $z = -n/2 + i \cdot y$ where $y \geq mn/\pi$.
For $0 \leq \lambda_{j} \leq n$ we have that
the real part of $z + \lambda_{j}$ lies in in the closed interval
$[-n/2,n/2]$.
Hence the argument of $z + \lambda_{j}$
is bounded by
$$
{\pi/2} - \frac{\pi/2}{m}
<
{\pi/2} - \arctan\left(\frac{n/2}{mn/\pi}\right)
\leq
\arg(z + \lambda_{j})
\leq
{\pi/2} + \arctan\left(\frac{n/2}{mn/\pi}\right)
<
{\pi/2} + \frac{\pi/2}{m}  ,
$$
where we used the inequality
$\arctan(\theta) < \theta$ for $\theta$ positive. 
Thus the argument of the product $\prod_{j=1}^{m} (z + \lambda_{j})$
is bounded by
$$
(m-1) \cdot {\pi/2}
<
\arg\left(\prod_{j=1}^{m} (z + \lambda_{j})\right)
<
(m+1) \cdot {\pi/2} .
$$
Hence for all partitions $\lambda$
the products
$\prod_{j=1}^{m} (z + \lambda_{j})$
all lie in the same open half-plane.
Therefore their sum, which is the box polynomial $B_{m,n}(z)$,
also lies in this open half-plane.
Thus $B_{m,n}(z)$
is non-zero, proving the upper bound.
The lower bound follows by complex conjugation.
\end{proof}

A different bound is obtained as follows.

\begin{theorem}
The roots $\{z_{j}\}$ of the box polynomial
$B_{m,n}(x)$
lie in the annulus
with inner radius~$n/2$
and
outer radius
$S(m+n,n) \cdot
(2/n)^{m-1} \cdot
\binom{m+n}{n}^{-1}$,
that is,
$$
n/2
\leq
|z_{j}|
\leq
\frac{S(m+n,n)}{(\frac{n}{2})^{m-1} \cdot \binom{m+n}{n}} .
$$
\label{theorem_roots_annulus}
\end{theorem}
\begin{proof}
The inner radius follows
since all roots have real part
$-n/2$ by Theorem~\ref{theorem_root_real_part}.
Let $z_{1}, z_{2}, \ldots, z_{m}$
be the roots of the box polynomial
$B_{m,n}(x)$. Then we know that
the product
$(-1)^{m} \cdot z_{1} z_{2} \cdots z_{m}$
is the ratio of the constant term
$S(m+n,n)$
over
the leading term
$\binom{m+n}{m}$, that is,
$S(m+n,n) \cdot \binom{m+n}{m}^{-1}$.
We obtain the upper bound as 
follows:
\[
|z_{j}|
=
\prod_{k \neq j} |z_{k}|^{-1}\cdot S(m+n,n) \cdot \binom{m+n}{m}^{-1}
\leq
(2/n)^{m-1} \cdot S(m+n,n) \cdot \binom{m+n}{m}^{-1}.
\qedhere
\]
\end{proof}

\begin{proposition}
The inner and outer radii of the annulus in
Theorem~\ref{theorem_roots_annulus}
are asymptotically equivalent
as $n$ tends to infinity,
that is,
$$
\frac{S(m+n,n)}{(\frac{n}{2})^{m-1} \cdot \binom{m+n}{n}} \sim n/2 .
$$
\end{proposition}
\begin{proof}
Note that the Stirling number of the second kind $S(m+n,n)$
is given by
\[
S(m+n,n)
=
\sum_{\substack{\lambda_{1},\dots,\lambda_{k}\geq 2\\
                            \sum_{i=1}^{k}\lambda_{i}=k+m}}
\binom{n+m}{\sum_{i=1}^{k} \lambda_{i}}
\cdot
p(\lambda_{1}, \ldots, \lambda_{k}) ,
\]
where $\lambda_{1},\dots,\lambda_{k}$ are the cardinalities of the non-singleton blocks
and
$p(\lambda_{1}, \ldots, \lambda_{k})$ does not depend on $n$.
As a polynomial in $n$, the only term in this expression with maximal degree corresponds to~$\lambda_{1}=\cdots=\lambda_{m}=2$.
This corresponds to counting set partitions into $m$ pairs and
$n-m$ singleton blocks,
of which there are $\binom{n+m}{2m} \cdot (2m-1)!!$.
Hence the Stirling number and the leading terms are
asymptotically equivalent, that is,
\begin{align*}
S(m+n,n)
& \sim
\binom{n+m}{2m} \cdot (2m-1)!!
\sim
\frac{n^{2m} \cdot (2m-1)!!}{(2m)!}
=
\frac{n^{2m}}{2^{m} \cdot m!}
\sim
(n/2)^{m} \cdot \binom{n+m}{m} ,
\end{align*}
where we used $m! \cdot \binom{n+c}{m} \sim n^{m}$ twice.
The last statement is equivalent to the proposition.
\end{proof}

\section{The excedance  matrix}
\label{section_excedance_matrix}
The {\em excedance algebra} is defined
as the quotient
\begin{align}
\mathbb{Z}\langle\av,\bv\rangle/(\bv\av-\av\bv-\av-\bv) .
\label{equation_excedance_algebra}
\end{align}
It was introduced by Clark and Ehrenborg~\cite{Clark_Ehrenborg}
and motivated by
Ehrenborg and Steingr{\'i}msson's
study of the excedance set statistic in~\cite{Ehrenborg_Steingrimsson}.
For a permutation $\pi=\pi_{1} \pi_{2} \cdots \pi_{n+1}$
in the symmetric group $\mathfrak{S}_{n+1}$
define its {\em excedance word} $u = u_{1} u_{2} \cdots u_{n}$
by $u_{j} = \bv$ if $\pi_{j} > j$
and $u_{j} = \av$ otherwise.
In other words, the letter $\bv$ encodes where the excedances occur in
the permutation.
Let the bracket~$[u]$ denote the number of permutations
in the symmetric group with excedance word $u$.
The bracket is the excedance set statistic and
it satisfies the recursion
$[u \cdot \bv\av \cdot v]
=
[u \cdot \av\bv \cdot v]
+
[u \cdot \av \cdot v]
+
[u \cdot \bv \cdot v]$;
see~\cite[Proposition~2.1]{Ehrenborg_Steingrimsson}.
This recursion is the motivation for the excedance algebra.
Also note that we have the initial conditions
that $[\av \cdot u] = [u \cdot \bv] = [u]$ and $[1] = 1$.

Consider the polynomial $E(m,n)$ given by the sum
of all $\ab$-words with exactly $m$~$\av$'s and $n$~$\bv$'s.
For instance,
$E(2,2)$ is given by
$\av\av\bv\bv+\av\bv\av\bv+\av\bv\bv\av+\bv\av\av\bv+\bv\av\bv\av+\bv\bv\av\av$.
After the quotient of equation~\eqref{equation_excedance_algebra}
every element in the excedance algebra
can be expressed in the standard basis
$\{\av^{i}\bv^{j}\}_{i,j \geq 0}$.
Let $c^{m,n}_{i,j}$ be the coefficient
of $\av^{i} \bv^{j}$
in the expansion of $E(m,n)$, that is,
$$
E(m,n)
= 
\sum_{\substack{0 \leq i \leq m \\ 0 \leq j \leq n}}
c^{m,n}_{i,j} \cdot \av^{i} \cdot \bv^{j} .
$$
Similarly for any polynomial $u$ in the excedance algebra,
define the coefficients
$c_{i,j}(u)$ by
$$
u
= 
\sum_{0 \leq i,j}
c_{i,j}(u) \cdot \av^{i} \cdot \bv^{j} .
$$
\begin{definition}
The \emph{excedance matrix} $M(m,n)$ is the
$(m+1) \times (n+1)$ matrix whose $(i,j)$ entry is $c^{m,n}_{i,j}$,
with rows and columns indexed from $0$ to $m$ and $0$ to $n$, respectively.
\end{definition}
\begin{example}
We have that
$M(2,2)$ is the matrix
$$ M(2,2)
=
\begin{pmatrix} 0&4&7\\4&14&12\\7&12&6 \end{pmatrix} , $$
since we have the expansion
$E(2,2)
=
6 \cdot \av\av\bv\bv + 
12 \cdot \av\av\bv +
7 \cdot \av\av +
12 \cdot \av\bv\bv + 
14 \cdot \av\bv +
4 \cdot \av +
7 \cdot \bv\bv + 
4 \cdot \bv$.
\end{example}
\begin{remark}
{\rm
\label{transpose}
Note that $M(m,n)=M(n,m)^T$ by Lemma 2.2 of~\cite{Ehrenborg_Steingrimsson} and the symmetry of the construction of $E(m,n)$.
}
\end{remark}

We now come to the connection between the excedance matrix
and the box polynomials.
\begin{proposition}
The box polynomial $B_{m,n}(x)$ is given by
$\sum_{j=0}^{m} c^{m,n}_{j,n} \cdot x^{j}$.
\end{proposition}
\begin{proof}
Since we are only interested in the last column of the excedance matrix,
we are only interested in terms with $n$ $\bv$'s.
In other words,
when replacing $\bv\av$ by $\av\bv+\av+\bv$
we can directly throw out the term $\av$.
That is, we replace
the relation with $\bv\av = \av\bv + \bv = (\av+1) \cdot \bv$.
Iterating this relation yields
\begin{align*}
E(m,n)
& = 
\sum_{p_{0}+p_{1}+\cdots+p_{n}=m}
\av^{p_{0}}
\cdot \bv \cdot
\av^{p_{1}}
\cdot \bv \cdots \bv \cdot
\av^{p_{n}} \\
& = 
\sum_{p_{0}+p_{1}+\cdots+p_{n}=m}
\av^{p_{0}}
\cdot 
(\av+1)^{p_{1}}
\cdots
(\av+n)^{p_{n}}
\cdot 
\bv^{n} .
\end{align*}
Now by applying the linear functional
$L(\av^{i} \bv^{n}) = x^{i}$
we have that
$L(E(m,n)) = B_{m,n}(x)$ by Definition~\ref{definition_box}.
\end{proof}

\begin{proposition}
The sum over all entries of the excedance matrix $M(m,n)$
is the Eulerian number $A(m+n+1,n+1)$.
\label{proposition_matrix_sum}
\end{proposition}
\begin{proof}
Note that the bracket $u \longmapsto [u]$ is a linear functional
on the excedance algebra.
Hence the bracket $[E(m,n)]$ enumerates the
number of permutations
in the symmetric group $\mathfrak{S}_{m+n+1}$
with $n$~excedances, that is, $A(m+n+1,n+1)$.
By expanding $E(m,n)$ into the standard basis
we have 
$[E(m,n)] = \sum_{i,j} c^{m,n}_{i,j} \cdot [\av^{i} \bv^{j}]$,
which is the sum of all the matrix entries since
$[\av^{i} \bv^{j}] = 1$.
\end{proof}

We apply Lemma 2.6 of~\cite{Clark_Ehrenborg}
to the sum of monomials $E(m,n)$ to obtain the following result.
\begin{lemma}
The alternating sums of the southwest to northeast diagonals
in the excedance matrix satisfy
$\sum_{i+j=k} (-1)^{i} \cdot c_{i,j}^{m,n} = 0$
for $k < m+n$.
Furthermore, the last entry is given by
$c_{m,n}^{m,n} = \binom{m+n}{m}$.
\label{lemma_diagonal_sums}
\end{lemma}
\begin{proof}
Lemma 2.6 of~\cite{Clark_Ehrenborg} states that if
$u$ is an $\ab$-word with $m$ $\av$'s and $n$ $\bv$'s,
then  $\sum_{i+j=k} (-1)^{i} \cdot c_{i,j}(u)=\delta_{m+n,0}$.
Summing this results over all such monomials
yields the result.
\end{proof}

We now give a recursion for the entries of the excedance matrix.
\begin{proposition}
The entries of the excedance matrix $M(m,n)$ satisfy
\[
c^{m,n}_{i,j}
=
c^{m,n-1}_{i,j-1}
+
\sum_{k=j}^{n} \binom{k}{j} \cdot c^{m-1,n}_{i-1,k}
+
\sum_{k=j}^{n} \binom{k}{j-1} \cdot c^{m-1,n}_{i,k} .
\]
\end{proposition}
\begin{proof}
One way to obtain the coefficient of $\av^{i} \bv^{j}$, or the entry $c^{m,n}_{i,j}$ of $M(m,n)$, is by post-multiplying monomials of the form $\av^{i}\bv^{j-1}$ by $\bv$, which yields the first term $c_{i,j-1}^{m,n-1}$ of the proposition. 

Note that a monomial of the form $\av^{i-1}\bv^{k}$ can yield $\av^{i} \bv^{j}$
for $k \geq j$ by post-multiplication by~$\av$.
As the $\av$ moves past each of the $k$ $\bv$'s at the end of $\av^{i-1}\bv^{k}$,
we choose $k-j$ of the $\bv\av$ pairs to become~$\av$,
and all other pairs become $\av\bv$.
This eliminates $k-j$ copies of $\bv$ and no copies of~$\av$,
leaving one term of the form $\av^{i}\bv^{j}$,
yielding the middle sum
$\sum_{k=j}^{n} \binom{k}{j} \cdot c_{i-1,k}^{m-1,n}$. 

Finally, we can obtain $\av^{i} \bv^{j}$ by post-multiplying a monomial
of the form $\av^{i} \bv^{k}$ by $\av$, for $k \geq j$.
Note that the power of $\av$ is the same in $\av^{i}\bv^{j}$ and $\av^{i}\bv^{k}$,
so, as we are post multiplying by $\av$,
we need to eliminate one copy of $\av$ and $k-j$ copies of $\bv$
as we move the $\av$ past the $k$ copies of $\bv$.
Eliminating the copy of $\av$ must be the last step,
so we choose one of the first $j$ $\bv$'s from the left to become an $\av$.
Suppose we choose the $l$'th $\bv$ to become an $\av$.
Of the remaining $k-l$ $\bv$'s to the right of the $l$'th $\bv$,
choose $k-j$ of them to become $\av$'s.
This yields the coefficient
$\sum_{l=1}^{j} \binom{k-l}{k-j} = \binom{k}{j-1}$
in the final sum of the proposition.
\end{proof}

We now make certain entries of the excedance matrix $M(m,n)$ explicit.

\begin{corollary}
\label{eulerian_entry}
The two entries
$c^{m,n}_{1,0}$ and $c^{m,n}_{0,1}$
of the excedance matrix $M(m,n)$
are given by the Eulerian number $A(m+n-1,n)$.
\end{corollary}
\begin{proof}
For a polynomial $v$ in the excedance algebra, observe that
when expanding $v \cdot \bv$ into the standard
basis, there is no $\av$ term,
that is, $c_{1,0}(v \cdot \bv) = 0$.
If we further assume that $v$ has
no constant term, 
we obtain $c_{1,0}(\av \cdot v) = 0$.
Finally,
Corollary~2.5 in~\cite{Clark_Ehrenborg}
states that
$c_{1,0}(\bv \cdot v \cdot \av) = [v]$.
(Note that their indexes are reversed,
that is, our $c_{i,j}(u)$ is their $c_{m-i,n-j}(u)$.)
Using the identity
\begin{align*}
E(m,n)
& =
\av \cdot E(m-2,n) \cdot \av
+ 
\av \cdot E(m-1,n-1) \cdot \bv \\
& +
\bv \cdot E(m-1,n-1) \cdot \av
+ 
\bv \cdot E(m,n-2) \cdot \bv ,
\end{align*}
and applying the linear functional $u \longmapsto c_{1,0}(u)$
we obtain
\begin{align*}
c_{1,0}(E(m,n))
=
c_{1,0}(\bv \cdot E(m-1,n-1) \cdot \av)
=
[E(m-1,n-1)] .
\end{align*}
This last expression enumerates the number of permutations
in the symmetric group
$\mathfrak{S}_{m+n-1}$ with $n-1$ excedances.
Finally, Lemma~\ref{lemma_diagonal_sums}
implies
$c^{m,n}_{1,0} = c^{m,n}_{0,1}$.
\end{proof}

By Proposition~\ref{proposition_closed_form_stirling_binomial}
we directly have
\begin{corollary}
The entries in the last column of the excedance matrix
$M(m,n)$ are given by
$c^{m,n}_{j,n} = \binom{m+n}{j} \cdot S(m+n-j,n)$, 
while the entries in the last row are given by~$c^{m,n}_{m,j} = \binom{m+n}{j} \cdot S(m+n-j,m)$.
\end{corollary}

For instance, the $(m-1,n-1)$ entry of the excedance matrix
is given by
\begin{align*}
c^{m,n}_{m-1,n-1}
& =
c^{m,n}_{m-2,n}
+
c^{m,n}_{m,n-2} \\
& =
\binom{m+n}{m-2} \cdot S(n+2,n)
+
\binom{m+n}{n-2} \cdot S(m+2,m) \\
& =
\binom{m+n}{m}
\cdot
\frac{m \cdot (m-1)}{(n+2) \cdot (n+1)}
\cdot
\left(3 \cdot \binom{n+2}{4} + \binom{n+2}{3} \right) \\
& +
\binom{m+n}{m}
\cdot
\frac{n \cdot (n-1)}{(m+2) \cdot (m+1)}
\cdot
\left(3 \cdot \binom{m+2}{4} + \binom{m+2}{3} \right) \\
& =
\binom{m+n}{m}
\cdot
m 
\cdot
n
\cdot 
\frac{3 m n - m - n}{12}.
\end{align*}

\section{Concluding remarks}

A number of questions and conjectures concerning the box polynomials and the excedance matrix remain unanswered.

\begin{question}
{\rm
Given a Schur function $s_{\lambda}(x_{1}, x_{2}, \ldots)$
what can be said about the properties of the one-variable
polynomial
$s_{\lambda}(x, x+1, \ldots, x+n)$?
For instance, are there any results on the location
of the roots?
}
\end{question}

\begin{question}
{\rm
As the box polynomial $B_{m,n}(x)$ is defined in 
terms of all partitions that fit in the $m$ by 
$n$ box, is there a Schubert calculus 
interpretation of the box polynomial? 
In other words, is the ring structure of the 
cohomology of the Grassmanian reflected in the 
algebraic properties of the box polynomials 
$B_{m,n}(x)$?
}
\end{question}

\begin{question}
{\rm
Let $L_{n}$ denote the graph known
as the cyclic ladder, that is, the product
of the cycle $C_{n}$ and the complete graph $K_{2}$.
The chromatic polynomial of $L_{n}$ is given by
$$
\chi(L_{n}; x)
  =
(x^{2} - 3 \cdot x + 3)^{n}
+
(x-1)
\cdot
\left(
(1-x)^{n} + (3-x)^{n}
\right)
+
x^{2} - 3 \cdot x + 1 ;
$$
see~\cite{Biggs_2001}
and~\cite{Biggs_Damerell_Sands_1972}.
Is there an explicit formula for the number of
partitions of the vertex set of $L_{n}$ into
$k$ blocks, which are independent sets,
that is, the number
$S(L_{n},k)$?
}
\end{question}

\begin{question}
{\rm
By combining Propositions~\ref{proposition_box_at_minus_r} and~\ref{proposition_odd_set_partitions}
when $n$ is even,
we know that the number of set partitions
of the set $[m+n]$ into $n$ blocks of odd
size is
$2^{m}$ times the number of set partitions
in standard form
of the set $[m+n/2]$ into
$n/2$ blocks such that the
minimal element of the $i$th block
has the same parity as $i$.
Is there a more combinatorial
proof of this fact,
for instance, a map where each
fiber has cardinality $2^{m}$?
}
\end{question}

\begin{question}
{\rm
Is there a sign-reversing involution proof of Proposition~\ref{corollary_T_m_n_stirling_numbers}? 
}
\end{question}

\begin{question}
{\rm
Is there a combinatorial proof of Theorem~\ref{theorem_root_real_part} using the partition interpretation of the box polynomial?
} 
\end{question}

Theorem~\ref{theorem_imaginary_bound} shows that
the imaginary part of the roots of
the box polynomial $B_{m,n}(x)$
is bounded above by
$O(m \cdot n)$.
However, computational data suggests
there is a sharper upper bound.

\begin{conjecture}
{\rm
The imaginary part of the roots of the box polynomial
$B_{m,n}(x)$ is bounded by $O(m \cdot \sqrt{n})$.
}
\end{conjecture}

A matrix is \emph{totally nonnegative} if the determinant of every square submatrix is nonnegative. 
\begin{conjecture}
{\rm Let $N(m,n)$ be the matrix obtained by flipping the excedance matrix $M(m,n)$ upside down, that is, its $i$th row is the $(m-i+1)$st row of $M(m,n)$.
Then the matrix $N(m,n)$ is totally nonnegative.
}
\end{conjecture}
Computational evidence supports this conjecture.
In fact, all of the determinants of square submatrices appear to be positive, except for the $1 \times 1$ matrix $c_{0,0}^{m,n} = 0$.

\begin{remark}
{\rm
Consider the polynomials whose roots come from columns of an excedance matrix other than the last; that is, polynomials $\sum_{j=0}^{m} c^{m,n}_{j,k} \cdot x^{j}$ for $k\neq n$. These polynomials have roots whose real parts are similar, though not equal like those of the roots of the box polynomial (see Figure~\ref{figure_column_box_roots}.) Furthermore, the smaller $k$ is, the smaller the roots are.
}
\end{remark}
\begin{figure}[ht]
\begin{center}
\includegraphics[scale=.2]{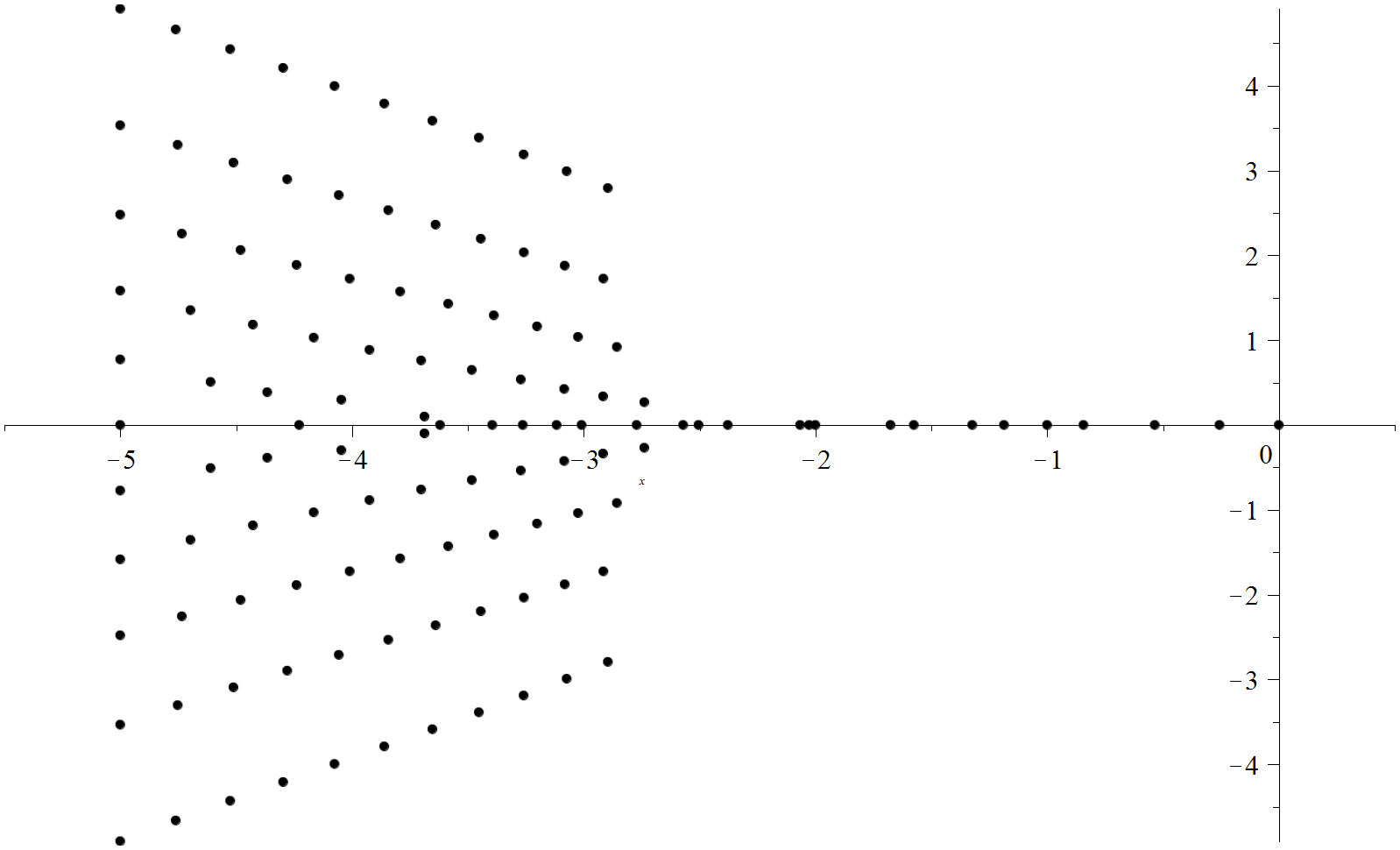}
\end{center}
\caption{The roots of the 10 polynomials whose coefficients are the columns of the excedance matrix $M(11,10)$, plotted in the complex plane.}
\label{figure_column_box_roots}
\end{figure}

\begin{question}
{\rm
Can the entire excedance matrix be characterized in terms of coefficients of polynomials obtained using the difference operators $\Delta_x$ and $\Delta_y$, just as its rightmost column, yielding the box polynomials, is defined by $\Delta_x$? Doing so could make each entry of the excedance matrix explicit.
}
\end{question}

\begin{question}
{\rm
Is there a way to prove that the Eulerian 
numbers are
unimodal using the excedance set statistic? One 
possible
approach is as follows. Let $\mathcal{E}(m,n)$ 
be the set of all
$\ab$-monomials with $m$ $\av$'s and $n$ 
$\bv$'s.
Is there an injective function
$\varphi : \mathcal{ E}(m,n) \longrightarrow 
\mathcal{E}(m+1,n-1)$
for all $m<n$ such that $[u] \leq [\varphi(u)]$?
If such a function $\varphi$ exists, the 
unimodality of the Eulerian numbers follows by 
summing over all monomials $u$ in 
$\mathcal{E}(m,n)$. 

One potential candidate function $\varphi(u)$ is defined by
factoring $u$ as $v \cdot w$, where $v$ has
exactly one more $\av$ than $\bv$'s.
Then let $\varphi(v \cdot w) = \overline{v}^{*} \cdot w$,
where $*$ reverses the word and the bar exchanges
$\av$'s and $\bv$'s.
This function works for small length words, but there is a counterexample at length~$22$,
namely:
\begin{align*}
u &= \bv^{5}\av\bv\av\bv\av^{5}\bv\av\bv\av\bv\av^{2} \cdot \av, \\
\varphi(u) &= 
\bv^{2}\av\bv\av\bv\av\bv^{5}\av\bv\av\bv\av^{5} \cdot \av,
\end{align*}
and $[u] =150803880738467413$ which
is greater than
$[\varphi(u)] = 150373062932169969$.
}
\end{question}

\section*{Acknowledgements}

The authors thank Richard Stanley for the reference~\cite{Callan}.
This work was supported by a grant from the Simons Foundation (\#429370, Richard Ehrenborg).

\bibliographystyle{plain}
\bibliography{bibliography.bib}

\bigskip

\noindent
{\em R.\ Ehrenborg, A.\ Happ,
Department of Mathematics,
University of Kentucky,
Lexington, KY 40506-0027,}
{\tt richard.ehrenborg@uky.edu},
{\tt alex.happ@uky.edu} \\

\noindent
{\em D.\ Hedmark,
Department of Mathematics,
Montgomery Bell Academy,
4001 Harding Road,
Nashville, TN 37205-1902,}
{\tt dustin.hedmark@montgomerybell.edu} \\

\noindent
{\em C.\ Hettle, 
School of Mathematics, 
Georgia Institute of Technology, 
Atlanta, GA 30332-0160,} \\
{\tt chettle@gatech.edu} \\

\end{document}